\documentclass[leqno,12pt]{amsart}
\usepackage{amsfonts}
\usepackage{amsmath,amssymb,amsthm}

\newcommand{\R}{\mathbb R}

\newcommand{\E}{\mathbb E}

\renewcommand{\span}{\mathrm{span}}
\newcommand{\tr}{\mathrm{tr}}

\newtheorem{thm}{Theorem}[section]

\theoremstyle{definition}
\newtheorem{defn}[thm]{Definition}
\theoremstyle{remark}

\newcommand{\ds}{\displaystyle}

\begin{document}

\title[Surfaces with Parallel Normalized Mean Curvature Vector Field] {Surfaces with Parallel Normalized Mean Curvature Vector Field in Euclidean or Minkowski 4-Space}

\author{Georgi Ganchev and Velichka Milousheva}
\address{Institute of Mathematics and Informatics, Bulgarian Academy of Sciences,
Acad. G. Bonchev Str. bl. 8, 1113 Sofia, Bulgaria}
\email{ganchev@math.bas.bg}
\address{Institute of Mathematics and Informatics, Bulgarian Academy of Sciences,
Acad. G. Bonchev Str. bl. 8, 1113, Sofia, Bulgaria} \email{vmil@math.bas.bg}


\begin{abstract}
We study surfaces with parallel normalized mean curvature vector field in Euclidean or Minkowski 4-space. On any such surface  we introduce special isothermal  parameters (canonical parameters) and describe  these surfaces in terms of three invariant functions. 
 We prove that any surface with parallel normalized mean curvature vector field parametrized by canonical parameters is  determined uniquely up to a motion in Euclidean (or  Minkowski) space by the three invariant functions  satisfying a system of three partial differential equations.  We find examples of surfaces with parallel normalized mean curvature vector field and solutions to the corresponding  systems of PDEs in  Euclidean  or Minkowski space  in the class of the meridian surfaces.
\end{abstract}

\maketitle

\section{Introduction}

A basic class of surfaces in Riemannian and pseudo-Riemannian geometry are surfaces with parallel mean curvature vector field, since they are critical points of some natural functionals and  play important role in differential geometry,  the theory of harmonic maps, as well as in physics. Surfaces with parallel mean curvature vector field in Riemannian space forms were classified in the early 1970s by Chen \cite{Chen1} and Yau  \cite{Yau}. Recently, spacelike surfaces with parallel mean
curvature vector field  in pseudo-Euclidean spaces with arbitrary codimension  were classified in \cite{Chen1-2}.
Lorentz surfaces with parallel mean curvature vector field in arbitrary pseudo-Euclidean space $\E^m_s$ are studied in \cite{Chen-KJM, Fu-Hou}. A survey on classical and recent results on submanifolds with parallel mean curvature vector in Riemannian manifolds
as well as in pseudo-Riemannian manifolds is presented in \cite{Chen-survey}.

The class of surfaces with parallel mean curvature vector field is naturally extended to the class of surfaces with parallel
normalized mean curvature vector field.
A submanifold in a Riemannian manifold is said to have parallel
normalized mean curvature vector field  if the mean curvature vector is non-zero and the
unit vector in the direction of the mean curvature vector is parallel in the normal
bundle  \cite{Chen-MM}.
It is well known that submanifolds with non-zero parallel mean curvature vector field also have parallel normalized mean curvature vector field.
 But the condition to have parallel normalized mean curvature vector field is  weaker than the condition to have parallel mean curvature vector field.
Every surface in the Euclidean 3-space has parallel normalized mean
curvature vector field but in the 4-dimensional Euclidean space, there exist examples of surfaces
with parallel normalized mean
curvature vector field, but  with non-parallel mean curvature vector field. In \cite{Chen-MM} it is
proved that every analytic surface with parallel normalized mean curvature
vector in the Euclidean space $\mathbb{E}^{m}$ must either lie in a
4-dimensional space $\mathbb{E}^{4}$ or in a hypersphere of $\mathbb{E}^{m}$
as a minimal surface.

Spacelike submanifolds with parallel normalized mean curvature vector field in a general de Sitter space are studied in \cite{Shu}.
 It is shown that compact spacelike submanifolds whose mean curvature does not vanish and whose corresponding normalized vector field is parallel, must be, under some suitable geometric assumptions, totally umbilical.

In \cite{AGM} we studied the local theory of Lorentz surfaces with parallel normalized mean curvature vector field in the pseudo-Euclidean space with neutral metric $\E^4_2$.  Introducing special geometric parameters (called canonical parameters) on each  such surface, we  described the Lorentz surfaces with parallel normalized mean curvature vector field in terms of three invariant functions satisfying a system of three partial differential equations. 

In the present paper we study surfaces with parallel normalized mean curvature vector field in the Euclidean 4-space $\E^4$ and the Minkowski 4-space $\E^4_1$. We introduce canonical parameters on each such surface that allow us to describe  these surfaces in terms of three invariant functions. We prove that any  surface with  parallel normalized mean curvature vector field is determined up to a motion in $\E^4$ by three functions $\lambda(u,v)$, $\mu(u,v)$ and $\nu(u,v)$ satisfying the following  system of partial differential equations 
\begin{equation} \label{E:Eq0-1} 
\begin{array}{l}
\vspace{2mm}
\nu_u = \lambda_v - \lambda (\ln|\mu|)_v;\\
\vspace{2mm}
\nu_v = \lambda_u - \lambda (\ln|\mu|)_u;\\
\vspace{2mm}
\nu^2 - (\lambda^2 + \mu^2) = \frac{1}{2}|\mu| \Delta \ln |\mu|,
\end{array} 
\end{equation}
where $\Delta$ denotes the Laplace operator.

The class of spacelike surfaces with parallel normalized mean curvature vector field in the Minkowski space $\E^4_1$ is described by three functions 
 $\lambda(u,v)$, $\mu(u,v)$ and $\nu(u,v)$ satisfying the following  system of partial differential equations 
\begin{equation}  \label{E:Eq0-2}
\begin{array}{l}
\vspace{2mm}
\nu_u = \lambda_v - \lambda (\ln|\mu|)_v;\\
\vspace{2mm}
\nu_v = \lambda_u - \lambda (\ln|\mu|)_u;\\
\vspace{2mm}
\varepsilon(\nu^2 - \lambda^2 + \mu^2) = \frac{1}{2}|\mu| \Delta \ln |\mu|,
\end{array} 
\end{equation}
where $\varepsilon = 1$ corresponds to the case the mean curvature vector field is spacelike, $\varepsilon = - 1$ corresponds to the case  the mean curvature vector field is timelike.

Examples of surfaces with parallel normalized mean curvature vector field  in $\E^4$ and $\E^4_1$  can be found in the class of the so-called meridian surfaces -- two-dimensional surfaces which are one-parameter systems of meridians of a  rotational hypersurface \cite{GM-BKMS, GM-MC, GM-GIQ}. In the Euclidean space $\E^4$ there is one type of meridian surfaces, while in the Minkowski space $\E^4_1$ we distinguish three types of meridian surfaces depending on the casual character of the rotational axis (spacelike, timelike, or lightlike). 
The functions $\lambda(u,v)$, $\mu(u,v)$, $\nu(u,v)$ of each meridian surface with parallel normalized mean curvature vector field parametrized by canonical parameters $(u,v)$ in $\E^4$ or $\E^4_1$ give a solution to the system of partial differential equations \eqref{E:Eq0-1} or \eqref{E:Eq0-2}, respectively.
In Section \ref{S:Ex} we give explicit examples of solutions to systems \eqref{E:Eq0-1} and  \eqref{E:Eq0-2}.

\section{Preliminaries} \label{S:Pre}

Let $M^2$ be a 2-dimensional surface in the Euclidean space $\E^4$ or a spacelike surface in the Minkowski space $\E^4_1$.  
We denote by $\nabla$ and $\nabla'$ the Levi Civita connections of $M^2$  and $\E^4$ (or $\E^4_1$), respectively.
For any tangent vector fields $x$, $y$  and any normal vector field $\xi$ of $M^2$, the formulas of Gauss and Weingarten are given respectively by \cite{Chen1}:
$$\begin{array}{l}
\vspace{2mm}
\nabla'_xy = \nabla_xy + \sigma(x,y);\\
\vspace{2mm}
\nabla'_x \xi = - A_{\xi} x + D_x \xi.
\end{array}$$
These formulas define the second fundamental form $\sigma$, the normal
connection $D$, and the shape operator $A_{\xi}$ with respect to
$\xi$.  The shape operator $A_{\xi}$ is a symmetric endomorphism of the tangent space $T_pM^2$ at $p \in M^2$.
  
The mean curvature vector  field $H$ of $M^2$ in $\E^4$ (or $\E^4_1$)
is defined as $$H = \ds{\frac{1}{2}\,  \tr\, \sigma}.$$

A normal vector field $\xi$ on $M^2$ is called \emph{parallel in the normal bundle} (or simply \emph{parallel}) if $D{\xi}=0$ holds identically \cite{Chen2}.
The surface  is said to have \emph{parallel mean curvature vector field} if its mean curvature vector $H$ satisfies $D H =0$ identically.

Surfaces for which the mean curvature vector $H$ is non-zero and  there exists a unit vector field $b$ in the direction of  $H$, such that $b$ is parallel in the normal bundle, are called surfaces with \textit{parallel normalized mean curvature vector field} \cite{Chen-MM}.
It is easy to see  that if $M^2$ is a surface  with non-zero parallel mean curvature vector field $H$ (i.e. $DH = 0$),
then $M^2$ is a surface with parallel normalized mean curvature vector field, but the converse is not true in general.
It is true only in the case $\Vert H \Vert = const$.

\section{Canonical parameters on surfaces with parallel normalized mean curvature vector field in $\E^4$}

Let $M^2: z = z(u,v), \, \, (u,v) \in {\mathcal D}$ (${\mathcal D}
\subset \R^2$) be a local parametrization of a surface free of minimal points in the Euclidean 4-space $\E^4$. In \cite{GM2} we defined principal lines and introduced a geometrically determined orthonormal frame field  $\{x, y, b, l\}$ at each point of the surface which is defined by the principal lines and the mean curvature vector field $H$.
With respect to this frame field  we have the following Frenet-type formulas:
\begin{equation} \label{E:Eq1}
\begin{array}{ll}
\vspace{1mm} \nabla'_xx=\quad \quad \quad \gamma_1\,y+\,\nu_1\,b;
& \qquad
\nabla'_xb=-\nu_1\,x-\lambda\,y\quad\quad \quad +\beta_1\,l;\\
\vspace{1mm} \nabla'_xy=-\gamma_1\,x\quad \quad \; + \; \lambda\,b
\; + \mu\,l;  & \qquad
\nabla'_yb=-\lambda\,x - \; \nu_2\,y\quad\quad \quad +\beta_2\,l;\\
\vspace{1mm} \nabla'_yx=\quad\quad \;-\gamma_2\,y \; + \lambda\,b
\; +\mu\,l;  & \qquad
\nabla'_xl= \quad \quad \quad \;-\mu\,y-\beta_1\,b;\\
\vspace{1mm} \nabla'_yy=\;\;\gamma_2\,x \quad\quad\quad+\nu_2\,b;
& \qquad \nabla'_yl=-\mu\,x \quad \quad \quad \;-\beta_2\,b,
\end{array}
\end{equation}
\noindent
where  $\gamma_1, \, \gamma_2, \, \nu_1,\, \nu_2, \,
\lambda, \, \mu, \, \beta_1$, $\beta_2$ are functions on $M^2$ determined by the
geometric frame field  as follows:
\begin{equation} \label{E:Eq2}
\begin{array}{l}
\vspace{2mm}
\nu_1 = \langle \nabla'_xx, b\rangle, \qquad \nu_2 = \langle \nabla'_yy, b\rangle, \qquad \, \lambda = \langle \nabla'_xy, b\rangle,
\qquad \mu = \langle \nabla'_xy, l\rangle,\\
\vspace{2mm}
\gamma_1 =  \langle \nabla'_xx, y\rangle,  \qquad  \gamma_2 =  \langle \nabla'_yy, x\rangle, \qquad \beta_1 = \langle \nabla'_xb, l\rangle, \qquad
\beta_2 = \langle \nabla'_yb, l\rangle.
\end{array}
\end{equation}
We call these functions \textit{geometric functions} of the surface since they determine the surface up to a rigid motion in $\E^4$.

We considered the general class of surfaces for which $\mu_u \,\mu_v \neq 0$ and for this class of surfaces we proved the fundamental existence and uniqueness theorem  in terms of their geometric functions.
The theorem states:

\begin{thm}\label{T:Fund-Theorem}
\cite{GM2}
Let $\gamma_{1},\,\gamma_{2},\,\nu_{1},\,\nu_{2},\,\lambda,\,\mu,\,\beta_{1},\beta_{2}$
be smooth functions, defined in a domain $\mathcal{D},\,\,\mathcal{D}\subset{\R}^{2}$,
and satisfying the conditions
\begin{equation*}
\begin{array}{l}
\vspace{2mm}
\displaystyle{\frac{\mu_{u}}{2\mu\,\gamma_{2}+\nu_{1}\,\beta_{2}-\lambda\,\beta_{1}}>0}; \qquad \displaystyle{\frac{\mu_{v}}{2\mu\,\gamma_{1}+\nu_{2}\,\beta_{1}-\lambda\,\beta_{2}}>0};\\
\vspace{2mm}
- \gamma_1 \sqrt{E} \sqrt{G} = (\sqrt{E})_v; \qquad \quad- \gamma_2 \sqrt{E} \sqrt{G} = (\sqrt{G})_u;\\
\vspace{2mm} \nu_1 \,\nu_2 - (\lambda^2 + \mu^2) =
\displaystyle{\frac{1}{\sqrt{E}}\,(\gamma_2)_u
+ \frac{1}{\sqrt{G}}\,(\gamma_1)_v - \left((\gamma_1)^2 + (\gamma_2)^2\right)};\\
\vspace{2mm} 2\lambda\, \gamma_2 + \mu\,\beta_1 - (\nu_1 -
\nu_2)\,\gamma_1 =
\displaystyle{\frac{1}{\sqrt{E}}\,\lambda_u - \frac{1}{\sqrt{G}}\,(\nu_1)_v};\\
\vspace{2mm} 2\lambda\, \gamma_1 + \mu\,\beta_2 + (\nu_1 -
\nu_2)\,\gamma_2 =
\displaystyle{ - \frac{1}{\sqrt{E}}\,(\nu_2)_u + \frac{1}{\sqrt{G}}\,\lambda_v};\\
\gamma_1\,\beta_1 - \gamma_2\,\beta_2 + (\nu_1 - \nu_2)\,\mu  =
\displaystyle{ - \frac{1}{\sqrt{E}}\,(\beta_2)_u +
\frac{1}{\sqrt{G}}\,(\beta_1)_v}.
\end{array}
\end{equation*}
where $\vspace{2mm} \sqrt{E}={\ds\frac{\mu_u}{2\mu\,\gamma_{2}+\nu_{1}\,\beta_{2}-\lambda\,\beta_{1}}},
\sqrt{G}={\ds\frac{\mu_v}{2\mu\,\gamma_{1}+\nu_{2}\,\beta_{1}-\lambda\,\beta_{2}}}$.
Let $\{x_0, \, y_0, \, b_0,\, l_0\}$ be an
orthonormal frame at a point $p_0 \in \R^4$. Then there exist a
subdomain ${\mathcal{D}}_0 \subset \mathcal{D}$ and a unique
surface $M^2: z = z(u,v), \,\, (u,v) \in {\mathcal{D}}_0$, passing
through $p_0$, such that $\{x_0, \, y_0, \, b_0,\, l_0\}$ is the
geometric frame of $M^2$ at the point $p_0$ and $\gamma_1, \, \gamma_2, \, \nu_1,\,
\nu_2, \, \lambda, \, \mu, \, \beta_1, \beta_2$ are the geometric
functions of $M^2$.
\end{thm}

Hence, any surface of the general class is determined up to a rigid motion  in $\E^4$ by the  eight geometric functions $\gamma_1, \, \gamma_2, \, \nu_1,\, \nu_2, \, \lambda, \, \mu, \, \beta_1$, $\beta_2$ satisfying some natural conditions.

\vskip 2mm
Now we focus our attention on the class of surfaces with parallel normalized mean curvature vector field.  
The mean curvature vector field $H$ is expressed as
\begin{equation}  \label{E:Eq3}
H = \ds{\frac{\nu_1 + \nu_2}{2}\, b}.
\end{equation}
Using \eqref{E:Eq3} and \eqref{E:Eq1} we obtain that 
\begin{equation}  \notag
\begin{array}{l}
\vspace{2mm} 
D_x H = \ds{x\left(\frac{\nu_1 + \nu_2}{2}\right)  b +\beta_1\,l};\\
\vspace{2mm} 
D_y H = \ds{y\left(\frac{\nu_1 + \nu_2}{2}\right)  b +\beta_2\,l}.
\end{array}
\end{equation}
Hence, the mean curvature vector field $H$ is parallel if and only if $\beta_1 = \beta_2 =0$ and $\nu_1 + \nu_2 = const$.
The normalized mean curvature vector field of $M^2$ is $b$. It follows from \eqref{E:Eq1}  that $b$ is parallel in the normal bundle if and only if $\beta_1 = \beta_2 =0$. 

We shall study surfaces  with parallel normalized mean curvature vector field, but  with non-parallel mean curvature vector field. 
They are characterized by the conditions 
$\beta_1 = \beta_2 =0$, $\nu_1 +\nu_2  \neq const$. For these surfaces we  shall introduce special isothermal parameters that allow us to formulate the fundamental existence and uniqueness theorem in terms of three geometric functions.

\begin{defn}\label{D:Canonocal parameters}
Let $M^2$ be a surface with parallel normalized mean
curvature vector field. The parameters $(u,v)$ of $M^2$ are said to be \textit{canonical}, if
$$E(u,v) = \ds{\frac{1}{|\mu (u,v)|}};\qquad F(u,v) = 0; \qquad G(u,v) = \ds{\frac{1}{|\mu (u,v)|}}.$$
\end{defn}

\begin{thm}\label{T:Canonocal parameters}
Each  surface with  parallel normalized mean
curvature vector field in $\E^4$ locally admits  canonical parameters. 
\end{thm}

\begin{proof}
Using the Gauss and Codazzi equations, from \eqref{E:Eq1} we obtain that the geometric functions $\gamma_1, \, \gamma_2, \, \nu_1,\, \nu_2, \,
\lambda, \, \mu, \, \beta_1$, $\beta_2$ of a surface free of minimal points satisfy the following integrability
conditions:
\begin{equation}  \label{E:Eq4}
\begin{array}{l}
\vspace{2mm}
2\mu\, \gamma_2 + \nu_1\,\beta_2 - \lambda\,\beta_1 = x(\mu);\\
\vspace{2mm}
2\mu\, \gamma_1 - \lambda\,\beta_2 + \nu_2\,\beta_1 = y(\mu);\\
\vspace{2mm}
2\lambda\, \gamma_2 + \mu\,\beta_1 - (\nu_1 - \nu_2)\,\gamma_1 = x(\lambda) - y(\nu_1);\\
\vspace{2mm}
2\lambda\, \gamma_1 + \mu\,\beta_2 + (\nu_1 - \nu_2)\,\gamma_2 = - x(\nu_2) + y(\lambda);\\
\vspace{2mm}
\gamma_1\,\beta_1 - \gamma_2\,\beta_2 + (\nu_1 - \nu_2)\,\mu  = -
x(\beta_2) + y(\beta_1); \\
\vspace{2mm}
\nu_1 \,\nu_2 - (\lambda^2 + \mu^2) = x(\gamma_2) + y(\gamma_1) - \left((\gamma_1)^2 + (\gamma_2)^2\right).
\end{array}
\end{equation}

Putting $\beta_1 = \beta_2 =0$ in formulas \eqref{E:Eq4}, we get
\begin{equation}  \label{E:Eq5}
\begin{array}{l}
\vspace{2mm}
2\mu\, \gamma_2 = x(\mu);\\
\vspace{2mm}
2\mu\, \gamma_1 = y(\mu);\\
\vspace{2mm}
2\lambda\, \gamma_2 - (\nu_1 - \nu_2)\,\gamma_1 = x(\lambda) - y(\nu_1);\\
\vspace{2mm}
2\lambda\, \gamma_1 + (\nu_1 - \nu_2)\,\gamma_2 = - x(\nu_2) + y(\lambda);\\
\vspace{2mm}
(\nu_1 - \nu_2)\,\mu = 0; \\
\vspace{2mm}
\nu_1 \,\nu_2 - (\lambda^2 + \mu^2) = x(\gamma_2) + y(\gamma_1) - \left((\gamma_1)^2 + (\gamma_2)^2\right).
\end{array}
\end{equation}

The first and second equalities of  \eqref{E:Eq5} imply $\gamma_1 = \ds{\frac{1}{2} y(\ln |\mu|)}; \,\, \gamma_2 = \ds{\frac{1}{2} x(\ln |\mu|)}$. On the other hand,  from \eqref{E:Eq2} it follows that $\gamma_1 = - y(\ln \sqrt{E}), \,\, \gamma_2 = - x(\ln \sqrt{G})$. Hence, $x(\ln|\mu| E) = 0$ and $y(\ln|\mu| G) = 0$, which imply that  
$E |\mu|$ does not depend on $v$, and $G |\mu|$ does not depend on $u$.
Hence,
there exist functions $\varphi(u) >0 $ and $ \psi(v) >0$, such that
$$E |\mu| = \varphi(u); \qquad G|\mu| = \psi(v).$$
Under the following change of the parameters:
\begin{equation*}  \label{E:Eq6}
\begin{array}{l}
\vspace{2mm}
\overline{u} = \ds{\int_{u_0}^u \sqrt{\varphi(u)}\, du}
+ \overline{u}_0, \quad \overline{u}_0 = const\\
[2mm]
\overline{v} = \ds{\int_{v_0}^v  \sqrt{\psi(v)}\, dv + \overline{v}_0},
\quad \overline{v}_0 = const
\end{array}
\end{equation*}
we obtain
$$\overline{E} = \ds{\frac{1}{|\mu|}}; \qquad \overline{F} = 0;
\qquad \overline{G} = \ds{\frac{1}{|\mu|}},$$
which imply that the  parameters $(\overline{u},\overline{v})$ are canonical.

\end{proof}

\vskip 3mm
\noindent
\textbf{Remark}: Since $\mu \neq 0$ for surfaces free of minimal points, from the fifth equality of \eqref{E:Eq5} we get that $\nu_1 = \nu_2$. Let us denote $\nu = \nu_1 = \nu_2$. 

\vskip 3mm
Now we suppose that $M^2: z = z(u,v), \,\, (u,v) \in {\mathcal D}$ is a surface with  parallel normalized mean
curvature vector field  parametrized by canonical parameters $(u,v)$.
It follows from the first two equalities in \eqref{E:Eq5}  that the functions $\gamma_1$ and $\gamma_2$
are expressed by:
$$\gamma_1 = \left(\sqrt{|\mu|}\right)_v; \qquad \gamma_2
= \left(\sqrt{|\mu|}\right)_u.$$
The third and fourth equalities of \eqref{E:Eq5} imply the following partial differential equations:
\begin{equation}  \notag
\begin{array}{l}
\vspace{2mm}
\nu_u = \lambda_v - \lambda (\ln|\mu|)_v;\\
\vspace{2mm}
\nu_v = \lambda_u - \lambda (\ln|\mu|)_u.\\
\end{array}
\end{equation}
The last  equality of  \eqref{E:Eq5}  implies 
\begin{equation}  \notag
\nu^2 - (\lambda^2 + \mu^2) = \frac{1}{2}|\mu| \Delta \ln |\mu|,
\end{equation}
where $\Delta = \ds{\frac{\partial^2}{\partial u^2} + \frac{\partial^2}{\partial v^2}}$
is the Laplace operator.

\vskip 3mm
Now we shall formulate the fundamental existence and uniqueness theorem for the class of
 surfaces  with  parallel normalized mean curvature vector field in terms of canonical parameters.

\begin{thm}\label{T:Fundamental Theorem}
Let $\lambda(u,v)$, $\mu(u,v)$ and $\nu(u,v)$ be  smooth functions, defined in a domain
${\mathcal D}, \,\, {\mathcal D} \subset {\R}^2$, and satisfying the conditions
\begin{equation} \label{E:Eq7}
\begin{array}{l}
\vspace{2mm}
\mu  \neq 0, \quad \nu \neq const;\\
\vspace{2mm}
\nu_u = \lambda_v - \lambda (\ln|\mu|)_v;\\
\vspace{2mm}
\nu_v = \lambda_u - \lambda (\ln|\mu|)_u;\\
\vspace{2mm}
\nu^2 - (\lambda^2 + \mu^2) = \frac{1}{2}|\mu| \Delta \ln |\mu|.
\end{array} 
\end{equation}
If $\{x_0, \, y_0, \, b_0,\, l_0\}$ is an orthonormal frame at
a point $p_0 \in \E^4$, then there exists a subdomain ${\mathcal D}_0 \subset {\mathcal D}$
and a unique surface
$M^2: z = z(u,v), \,\, (u,v) \in {\mathcal D}_0$  with parallel normalized mean curvature vector field, such that $M^2$ passes through $p_0$, $\{x_0, \, y_0, \, b_0,\, l_0\}$ is the geometric
frame of $M^2$ at the point $p_0$, and  the functions   $\lambda(u,v)$, $\mu(u,v)$, $\nu(u,v)$ are the geometric functions
of $M^2$.
Furthermore, $(u,v)$ are canonical parameters of $M^2$.
\end{thm}

So, by introducing canonical parameters on a surface with  parallel normalized mean curvature vector field we manage to reduce up to three the number of functions and the number of partial differential equations which determine the surface up to a motion.

Our approach to the study of surfaces with parallel normalized mean curvature vector field  can be  applied also to spacelike surfaces in the Minkowski space $\E^4_1$. 

\section{Canonical parameters on spacelike surfaces with parallel normalized mean curvature vector field in $\E^4_1$}

Analogously to the theory of surfaces in the Euclidean space $\E^4$ we  developed an invariant local theory of spacelike surfaces in the Minkowski space $\E^4_1$ \cite{GM3}. 
We  introduced  principal lines and a geometrically determined moving frame field at each point of a spacelike surface. Writing derivative formulas of Frenet-type for this frame field, we obtained eight geometric
functions and  proved a fundamental existence and uniqueness theorem, stating
that any spacelike surface whose mean curvature vector at any
point is a non-zero spacelike vector or a timelike vector is
determined up to a motion in $\E^4_1$ by its eight geometric
functions satisfying some natural conditions \cite{GM3}.

Similarly to the Euclidean case,  we can introduce canonical parameters for the class of spacelike surfaces with parallel normalized mean curvature vector field in $\E^4_1$. The canonical parameters are special isothermal parameters satisfying the conditions 
$$E(u,v) = G(u,v)  = \ds{\frac{1}{|\mu (u,v)|}};\quad F(u,v) = 0.$$
Analogously to the proof of Theorem \ref{T:Canonocal parameters}, we get the following result.

\begin{thm}\label{T:Canonocal parameters-2}
Each  spacelike surface with  parallel normalized mean
curvature vector field in $\E^4_1$ locally admits  canonical parameters. 
\end{thm}

The fundamental existence and uniqueness theorem in terms of canonical parameters for the class of
spacelike  surfaces  with  parallel normalized mean curvature vector field states as follows.

\begin{thm}\label{T:Fundamental Theorem-2}
Let $\lambda(u,v)$, $\mu(u,v)$ and $\nu(u,v)$ be smooth functions, defined in a domain
$\mathcal{D}, \,\, \mathcal{D} \subset {\R}^2$, and satisfying the
conditions
\begin{equation}  \label{E:Eq7-1}
\begin{array}{l}
\vspace{2mm}
\mu  \neq 0, \quad \nu \neq const;\\
\vspace{2mm}
\nu_u = \lambda_v - \lambda (\ln|\mu|)_v;\\
\vspace{2mm}
\nu_v = \lambda_u - \lambda (\ln|\mu|)_u;\\
\vspace{2mm}
\varepsilon(\nu^2 - \lambda^2 + \mu^2) = \frac{1}{2}|\mu| \Delta \ln |\mu|, \quad \varepsilon = \pm 1.
\end{array} 
\end{equation}
If $\{x_0, \, y_0, \, b_0,\, l_0\}$ is an
orthonormal frame at a point $p_0 \in \E^4_1$ (with $\langle b_0, b_0 \rangle = \varepsilon$;
$\langle l_0, l_0 \rangle = -\varepsilon$), then there exist a
subdomain ${\mathcal{D}}_0 \subset \mathcal{D}$ and a unique
spacelike surface $M^2: z = z(u,v), \,\, (u,v) \in {\mathcal{D}}_0$ with parallel normalized mean curvature vector field, whose mean curvature vector at any point is spacelike (resp. timelike) in the case $\varepsilon = 1$ (resp. $\varepsilon = - 1$). Moreover, $M^2$ passes through $p_0$, $\{x_0, \, y_0, \, b_0,\, l_0\}$ is the geometric
frame of $M^2$ at the point $p_0$, and the functions   $\lambda(u,v)$, $\mu(u,v)$, $\nu(u,v)$  are the geometric functions
of $M^2$. Furthermore, $(u,v)$ are canonical parameters of $M^2$.

\end{thm}

\section{Examples} \label{S:Ex}

In \cite{GM2}  we constructed a special class of surfaces which are one-parameter systems of meridians of a rotational hypersurface in $\E^4$ and called them meridian surfaces. Each meridian surface $\mathcal{M}$ is determined by a meridian curve $m$ of a rotational
hypersurface in $\E^4$  and a smooth curve $c$ lying on the unit 2-dimensional sphere $\mathbb{S}^2(1)$ in a 3-dimensional Euclidean subspace $\E^3 \subset \E^4$. All invariants of the meridian surface  $\mathcal{M}$ are expressed by the curvature $\varkappa_m (u) $ of the meridian curve $m$ and the spherical curvature $\varkappa (v)$  of the curve $c$  on $\mathbb{S}^2(1)$. We classified the meridian surfaces with
constant Gauss curvature and the meridian surfaces with constant mean curvature. 
In \cite{GM-BKMS} we gave the complete classification of Chen meridian surfaces and meridian surfaces
with parallel normalized mean curvature vector field.
Meridian surfaces in the Minkowski space $\E^4_1$ are studied in \cite{GM-MC} and \cite{GM-GIQ}. Since in $\E^4_1$ there are three types of rotational hypersurfaces, namely rotational hypersurface with  timelike, spacelike, or lightlike axis, we distinguish three types of meridian surfaces -- elliptic, hyperbolic, and parabolic. We found all meridian surfaces of elliptic, hyperbolic, or parabolic type with parallel normalized mean curvature vector field.  
The geometric functions $\lambda(u,v)$, $\mu(u,v)$, $\nu(u,v)$ of a meridian surface with parallel normalized mean curvature vector field parametrized by canonical parameters $(u,v)$ in $\E^4$ (resp. $\E^4_1$) give a solution to the system of partial differential equations 
\eqref{E:Eq7} (resp. \eqref{E:Eq7-1}).

\subsection{A solution to the system of PDEs describing the  surfaces with parallel normalized mean curvature vector field in $\E^4$}

Let $Oe_1 e_2 e_3 e_4$ be the standard orthonormal frame  in $\E^4$. Let
$$f(u)= \sqrt{u^2 + 2u +5}; \quad g(u) = 2 \ln (u+1 +\sqrt{u^2 + 2u +5})$$
 and  consider the rotational hypersurface $M^3$ 
obtained by the rotation of the meridian curve $m: u
\rightarrow (f(u), g(u))$ about the $Oe_4$-axis, which is parametrized as follows:
$$M^3: Z(u,w^1,w^2) = f(u) \cos w^1 \cos w^2 e_1 +  f(u)\cos w^1 \sin w^2 e_2 + f(u) \sin w^1 e_3 + g(u) e_4.$$
If we denote by $l(w^1,w^2) = \cos w^1 \cos w^2 \,e_1 + \cos w^1 \sin w^2 \,e_2 + \sin w^1 \,e_3$ the unit position vector of the
2-dimensional sphere $\mathbb{S}^2(1)$ lying in $\E^3 = \span \{e_1, e_2, e_3\}$ and centered at the origin $O$, then the parametrization of $M^3$ is written shortly as
 $$M^3: Z(u,w^1,w^2) = f(u)\,l(w^1,w^2) + g(u) \,e_4.$$
If $w^1 = w^1(v)$, $w^2=w^2(v), \,\, v \in J, \, J \subset \R$, then $c: l = l(v) = l(w^1(v),w^2(v)), \, v \in J$ is a
smooth curve  on the sphere $\mathbb{S}^2(1)$. 
We consider the two-dimensional  surface $\mathcal{M}$  defined by:
\begin{equation*} 
\mathcal{M}: z(u,v) = f(u) \, l(v) + g(u)\, e_4, \quad u \in I, \, v \in J.
\end{equation*}
It is a one-parameter system of meridians of the rotational hypersurface 
$M^3$. We call $\mathcal{M}$ a meridian surface on $M^3$.

According to a result in \cite{GM-BKMS}, for an arbitrary spherical curve $c: l = l(v)$  on $\mathbb{S}^2(1)$ with spherical curvature 
$\varkappa (v) \neq 0$, the  corresponding meridian surface $\mathcal{M}$ is a surface with  parallel normalized mean curvature vector field but non-parallel mean curvature vector field. 
The geometric functions  $\lambda$, $\mu$, $\nu$ of the meridian surface $\mathcal{M}$ with respect to the parameters $(u,v)$ are:
\begin{equation}  \label{E:Eq-n}
\begin{array}{l}
\vspace{2mm}
\lambda (u,v) = \ds{\frac{\varkappa (v)}{2\sqrt{u^2 + 2u +5}}}; \\
\vspace{2mm}
\mu (u,v) = \ds{\frac{2}{u^2 + 2u +5}}; \\
\vspace{2mm}
\nu (u,v) = \ds{\frac{\varkappa (v)}{2\sqrt{u^2 + 2u +5}}}. 
\end{array}
\end{equation}
The mean curvature vector field is: 
$$H = \frac{\varkappa (v)}{2\sqrt{u^2 + 2u +5}} \, H_0,$$
where $H_0$ is a unit vector field in the direction of $H$. Since $\langle H, H \rangle \neq const$,  $\mathcal{M}$ is a surface with 
parallel normalized mean curvature vector  but non-parallel $H$.

It is important to note that the parameters $(u,v)$ coming from the parametrization of the meridian curve $m$ are not canonical parameters of the meridian surface $\mathcal{M}$. But if we change the parameters as follows
\begin{equation*} 
\begin{array}{l}
\vspace{2mm}
\bar{u} = \ln (u+1 +\sqrt{u^2 + 2u +5}) +v;\\ 
\vspace{2mm}
\bar{v} = - \ln (u+1 +\sqrt{u^2 + 2u +5}) + v,\\
\end{array}
\end{equation*}
then we get canonical parameters  $(\bar{u},\bar{v})$ of $\mathcal{M}$. Hence, changing  the parameters $(u,v)$ with $(\bar{u},\bar{v})$, we obtain that the functions  $\lambda(u(\bar{u},\bar{v}), v(\bar{u},\bar{v}))$, $\mu(u(\bar{u},\bar{v}), v(\bar{u},\bar{v}))$, and $\nu(u(\bar{u},\bar{v}), v(\bar{u},\bar{v}))$,  given by \eqref{E:Eq-n}
give a solution to the following system of PDEs:
\begin{equation}  \label{E:Eq7-n}
\begin{array}{l}
\vspace{2mm}
\nu_{\bar{u}} = \lambda_{\bar{v}} - \lambda (\ln|\mu|)_{\bar{v}};\\
\vspace{2mm}
\nu_{\bar{v}} = \lambda_{\bar{u}} - \lambda (\ln|\mu|)_{\bar{u}};\\
\vspace{2mm}
\nu^2 - (\lambda^2 + \mu^2) = \frac{1}{2}|\mu| \Delta \ln |\mu|.
\end{array} 
\end{equation}
One can see  also  by a direct computation that the functions given by \eqref{E:Eq-n} satisfy the equalities in system  \eqref{E:Eq7-n}.

\subsection{A solution to the system of PDEs describing the spacelike surfaces with parallel normalized mean curvature vector field in $\E^4_1$}

\vskip 3mm
Solutions to  system \eqref{E:Eq7-1} can be found in the class of the meridian surfaces in the Minkowski 4-space.
In $\E^4_1$ there exist three types of spacelike meridian surfaces and all of them give solutions to the corresponding system of PDEs.
Bellow we present a solution obtained from the class of the meridian surfaces lying on a rotational hypersurface with  lightlike axis, since it is the most interesting rotation in $\E^4_1$.

Let $Oe_1 e_2 e_3 e_4$ be the standard orthonormal frame  in $\E^4_1$, i.e. $\langle e_1,e_1 \rangle
=\langle e_2,e_2 \rangle =\langle e_3,e_3 \rangle = 1; \langle e_4,e_4 \rangle = -1$. 
Let $$f(u)= \sqrt{u+ 1}; \quad g(u) = -\frac{2}{3} (u+1)^{\frac{3}{2}}, \quad u \in (-1; + \infty)$$
 and  consider the rotational hypersurface with  lightlike axis  parametrized as follows:
\begin{equation*}
Z(u,w^1,w^2) =  f(u)\, w^1 \cos w^2 \,e_1 +  f(u)\, w^1 \sin w^2 \,e_2+ \left(f(u) \frac{(w^1)^2}{2} + g(u)\right)\xi_1 + f(u) \,\xi_2,
\end{equation*}
where $\displaystyle{\xi_1= \frac{e_3 + e_4}{\sqrt{2}}},\,\, \displaystyle{\xi_2= \frac{ - e_3 + e_4}{\sqrt{2}}}$.
According to a result in \cite{GM-GIQ}, for an arbitrary curve $c: l = l(v)= l(w^1(v),w^2(v)), \,\, v \in J, \, J \subset \R$ with  curvature $\varkappa (v) \neq 0$ lying  on the paraboloid 
$$\mathcal{P}^2: z(w^1,w^2) =  w^1 \cos w^2 \,e_1 +   w^1 \sin w^2 \,e_2+ \frac{(w^1)^2}{2} \, \xi_1 + \xi_2,$$
the corresponding meridian surface parametrized by
$$\mathcal{M}''': z(u,v) = f(u)\,l(v) + g(u) \xi_1$$ 
 is a surface with  parallel normalized mean curvature vector field (and non-parallel $H$). Indeed, the geometric functions $\lambda(u,v)$, $\mu(u,v)$, $\nu(u,v)$  of the meridian surface $\mathcal{M}'''$  with respect to the parameters $(u,v)$ look like:
\begin{equation} \label{E:Eq-geom-2}
\begin{array}{l}
\vspace{2mm}
\lambda (u,v) = \ds{\frac{\varkappa (v)}{2\sqrt{u+1}}}; \\
\vspace{2mm}
\mu (u,v) = \ds{-\frac{1}{2(u+1)}}; \\
\vspace{2mm}
\nu (u,v) = \ds{\frac{\varkappa (v)}{2\sqrt{u+1}}}.
\end{array}
\end{equation}
The mean curvature vector field is 
$$H = \frac{\varkappa (v)}{2\sqrt{u+1}} \, H_0,$$
where $H_0$ is a unit vector field in the direction of $H$. In this example,  $H_0$ is a spacelike vector at each point, and hence the solution corresponds to the case $\varepsilon = 1$ in system \eqref{E:Eq7-1}. 

Again, the parameters $(u,v)$ coming from the para\-met\-ri\-zation of the meridian curve $m$ are not canonical parameters of the meridian surface 
$\mathcal{M}'''$. But if we  change the parameters as follows:
\begin{equation*} 
\begin{array}{l}
\vspace{2mm}
\bar{u} = \ds{\sqrt{u+1} +\frac{v}{2}};\\ 
\vspace{2mm}
\bar{v} = \ds{- \sqrt{u+1} +\frac{v}{2}},\\
\end{array}
\end{equation*}
then we obtain  canonical parameters  $(\bar{u},\bar{v})$ of the meridian surface. Hence, the functions  $\lambda(u(\bar{u},\bar{v}), v(\bar{u},\bar{v}))$, $\mu(u(\bar{u},\bar{v}), v(\bar{u},\bar{v}))$, and $\nu(u(\bar{u},\bar{v}), v(\bar{u},\bar{v}))$, defined by  \eqref{E:Eq-geom-2}
give a solution to the following system of PDEs:
\begin{equation}  \label{E:Eq7-n-2}
\begin{array}{l}
\vspace{2mm}
\nu_{\bar{u}} = \lambda_{\bar{v}} - \lambda (\ln|\mu|)_{\bar{v}};\\
\vspace{2mm}
\nu_{\bar{v}} = \lambda_{\bar{u}} - \lambda (\ln|\mu|)_{\bar{u}};\\
\vspace{2mm}
\nu^2 - \lambda^2 + \mu^2 = \frac{1}{2}|\mu| \Delta \ln |\mu|.
\end{array} 
\end{equation}

It can be seen  also  by a direct computation that the functions  defined by  \eqref{E:Eq-geom-2}, satisfy the equalities in system  \eqref{E:Eq7-n-2}.

This is the system of PDEs describing the spacelike surfaces with parallel normalized mean curvature vector field in the Minkowski space $\E^4_1$ in the case $\varepsilon = 1$.
Solutions to the same system in the case $\varepsilon = -1$ can be found in the class of the meridian surfaces lying on rotational hypersurfaces with timelike axis.

\vskip 3mm \textbf{Acknowledgements:}
The  authors are partially supported by the National Science Fund,
Ministry of Education and Science of Bulgaria under contract DN 12/2.

\end{document}